\documentclass[11pt]{amsart}

\usepackage{amsmath, amssymb, amscd, cancel, graphicx, soul, stmaryrd}
\usepackage[T1]{fontenc}
\usepackage{mathdots}
\usepackage[utf8]{inputenc}
\usepackage{subcaption} 

\usepackage[dvipsnames,usenames]{color}
\usepackage[colorlinks=true, urlcolor=myblue, linkcolor=myblue, citecolor=myblue]{hyperref}
\usepackage{amsfonts}
\usepackage{enumerate}
\usepackage{paralist}
\usepackage{xypic}
\usepackage{hyperref}
\usepackage{color}
\usepackage{marginnote}
\usepackage{mathrsfs}

\definecolor{myblue}{RGB}{59,86,140}

\hypersetup{
linkbordercolor={1 0 0}, 
citebordercolor={0 1 0} 
}
\newtheorem{thm}{Theorem}[section]

\newtheorem{lem}[thm]{Lemma}
\newtheorem{prop}[thm]{Proposition}

\newtheorem{defin}[thm]{Definition}

\begin{document}
\bibliographystyle{alpha}

\title[]%
{The Conway knot is not slice}

\author{Lisa Piccirillo}
\address{Department of Mathematics, University of Texas, Austin, TX 78712}
\email{lpiccirillo@math.utexas.edu}

\maketitle

\begin{abstract}
A knot is said to be slice if it bounds a smooth properly embedded disk in $B^4$. We demonstrate that the Conway knot, $11n34$ in the Rolfsen tables, is not slice. This completes the classification of slice knots under 13 crossings, and gives the first example of a non-slice knot which is both topologically slice and a positive mutant of a slice knot.\\
\end{abstract}

\section{Introduction}
The classical study of knots in $S^3$ is 3-dimensional; a knot is defined to be trivial if it bounds an embedded disk in $S^3$. Concordance, first defined by Fox in \cite{Fox62}, is a 4-dimensional extension; a knot in $S^3$ is trivial in concordance if it bounds an embedded disk in $B^4$. In four dimensions one has to take care about what sort of disks are permitted. A knot is slice if it bounds a smoothly embedded disk in $B^4$, and topologically slice if it bounds a locally flat disk in $B^4$. There are many slice knots which are not the unknot, and many topologically slice knots which are not slice. 

It is natural to ask how characteristics of 3-dimensional knotting interact with concordance and questions of this sort are prevalent in the literature. Modifying a knot by positive mutation is particularly difficult to detect in concordance; we define positive mutation now. 

A Conway sphere for an oriented knot $K$ is an embedded $S^2$ in $S^3$ that meets the knot transversely in four points. The Conway sphere splits $S^3$ into two 3-balls, $B_1$ and $B_2$, and $K$ into two tangles $K_{B_1}$ and $K_{B_2}$. Any knot $K^*$  obtained from $K_{B_1}$ and $K_{B_2}$ after regluing $B_1$ to $B_2$ via an involution of the Conway sphere is called a mutant of $K$. If $K^*$ inherits a well-defined orientation from that of $K_{B_1}$ and $K_{B_2}$ then $K^*$ is a positive mutant of $K$. See Figure \ref{fig:conwayandKT} for an example.

Positive mutation preserves many three dimensional invariants of a knot, including the Alexander, Jones, Homfly, and Kauffman polynomials \cite{MT90}, the S-equivalence class \cite{KL01}, hyperbolic volume \cite{Rub87}, and 2-fold branched cover. Other powerful invariants are conjectured to be preserved by mutation, such as Khovanov homology \cite{Bar05} and the $\delta$-graded knot Floer groups \cite{BL12}. Studying the sliceness of knots which arise as a positive mutant of a slice knot is even trickier: all abelian and all but the subtlest metabelian sliceness obstructions vanish for such a knot, Rasmussen's s-invariant is conjectured to vanish \cite{Bar02}, and it is unknown whether any Heegaard Floer sliceness obstructions can detect such a knot. 

\begin{figure}
\includegraphics[width=10cm]{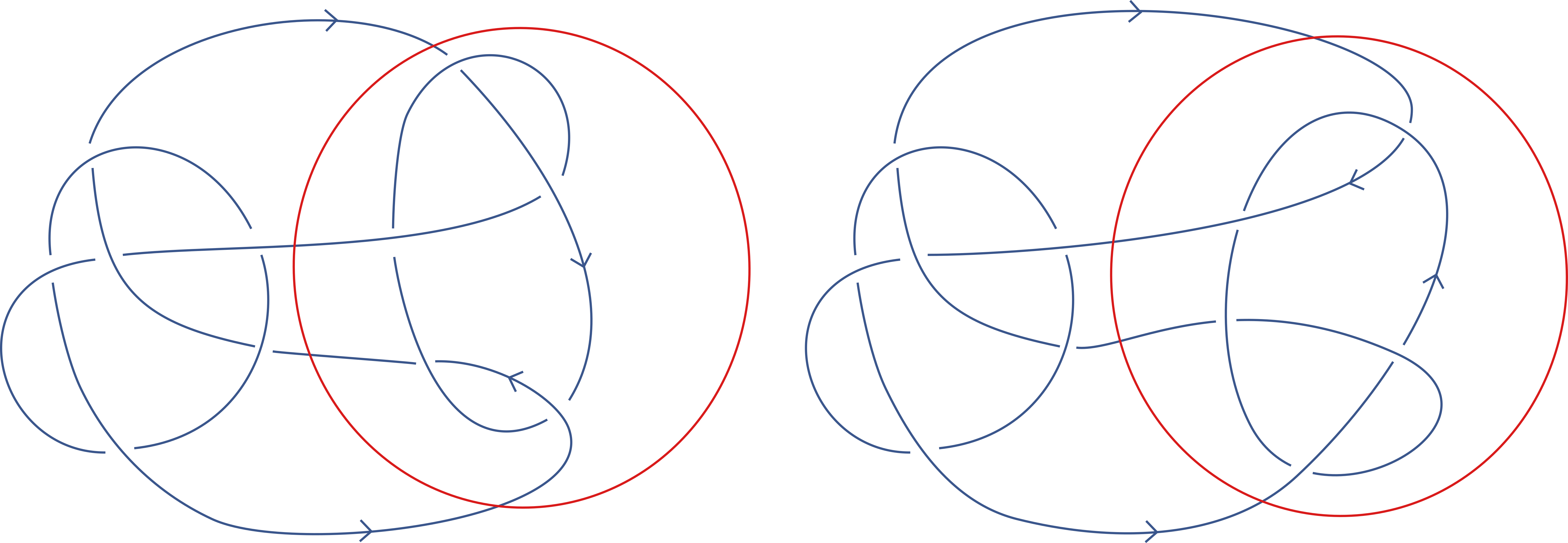}
\caption{Positive mutation from the Conway knot $C$ to the Kinoshita-Teresaka knot}
\label{fig:conwayandKT}
\end{figure}

The smallest pair of positive mutant knots, the 11 crossing Conway knot $C$ and Kinoshita-Teresaka knot, were discovered by Conway in \cite{Con69}; see Figure \ref{fig:conwayandKT}. These knots are also the smallest nontrivial knots with Alexander polynomial 1, hence all of their abelian and metabelian sliceness obstrucitons vanish, and by Freedman \cite{Fre83} both knots are topologically slice. The Conway and Kinoshita-Teresaka knots were first distinguished in isotopy by Riley in \cite{Ril71} via carefull study of their groups and later their Seifert genera were distinguished by Gabai \cite{Gab86}. One can readily show that the Kinoshita-Teresaka knot is slice. Despite the wealth of sliceness obstructions constructed in the past 20 years which are not known to be mutation invariant and do not necessarily vanish for Alexander polynomial 1 knots, it has remained open whether the Conway knot is slice. 

In 2001 Kirk and Livingston gave the first examples of non slice knots which are positive mutants of slice knots \cite{KL01}, and other examples have appeared since \cite{KL05, HKL10, Mil17}. All of these works rely on careful analysis of metabelian sliceness obstructions. Since metabelian obstructions obstruct topological sliceness, these techniques cannot detect any topologically slice knot which is a mutant of a slice knot, and are especially poorly suited to an Alexander polynomial 1 knot such as the Conway knot.

\begin{thm}
The Conway knot is not slice
\end{thm}

This completes the classification of slice knots of under 13 crossings \cite{CL05} and gives the first example of a non-slice knot which is both a positive mutant of a slice knot and topologically slice. 

Since there are no sliceness obstructions known to be well suited to detecting the Conway knot, we abandon the Conway knot. We will instead construct a knot $K'$ such that the Conway knot is slice if and only if $K'$ is slice, and concern ourselves with detecting the sliceness of this $K'$. (Experts will note that our $K'$ is not concordant to the Conway knot.) To construct such a $K'$ we will be interested in the following manifold. 

\begin{defin}
A knot trace $X(K)$ is a four manifold obtained by attaching a 0-framed 2-handle to $B^4$ with attaching sphere $K$.
\end{defin}
\noindent We will use $\cong$ to denote diffeomorphisms of manifolds. The following observation is folklore, for an early use see \cite{KM78}.
\begin{lem}{}
$K$ is slice if and only if $X(K)$ smoothly embeds in $S^4$.
\label{lem:slicetrace}
\end{lem}
\begin{proof}
For the `only if' direction: Consider $S^4$ and a smooth $S^3$ therein which decomposes $S^4$ into the union of two 4-balls $B_1$ and $B_2$. Consider $K$ sitting in this $S^3$. Since $K$ is slice, we can find a smoothly embedded disk $D_K$ which $K$ bounds in $B_1$. Observe now that $B_2\cup \overline{ \nu (D_K)} \cong X(K)$ is smoothly embedded in $S^4$. 

\noindent For the `if' direction: Let $F:S^2\to X(K)$ be a piecewise linear embedding such that the image of $F$ consists of the union of the cone on $K$ with the core of the two handle. Notice that $F$ is smooth away from the cone point $p$. Let $i:X(K)\to S^4$ be a smooth embedding. Then $(i\circ F)$ is a piecewise linear embedding of $S^2$ in $S^4$, which is smooth away from $i(p)$. Note that $W:=S^4\smallsetminus\nu(i(p))\cong B^4$ and that the restriction of $(i \circ F)$ to the complement of a small neighborhood of $F^{-1}(p)$ in $S^2$ is a smooth embedding of $D^2$ in $W \cong B^4$. Further, if we choose this neighborhood to be the inverse image of a sufficiently small neighborhood of $i(p)$ we have that $(i \circ F)(D^2 \smallsetminus \nu(F^{-1}(p)))$ intersects $\partial W$ in the knot $K$. 
\end{proof}

Then for any knots $K$ and $K'$ with $X(K)\cong X(K')$, $K$ is slice if and only if $K'$ is slice. There exists a small body of literature on producing pairs of knots $K$ and $K'$ with $X(K)\cong X(K')$, we will rely here on the dualizable patterns construction which was pioneered by Akbulut in \cite{Akb77}, developed by Lickorish \cite{Lic79} and Gompf-Miyazaki \cite{GM95}, and recently re-interpreted by the author \cite{Pic18}. Unknotting number one knots fit into this construction (a related statement appears in \cite{AJOT13}); using this we prove the following.
\begin{prop}\label{Prop:J'}
The knot $K'$ in Figure \ref{fig:J'} has $X(C)\cong X(K')$. 
\end{prop}
\begin{figure}
\includegraphics[width=5cm]{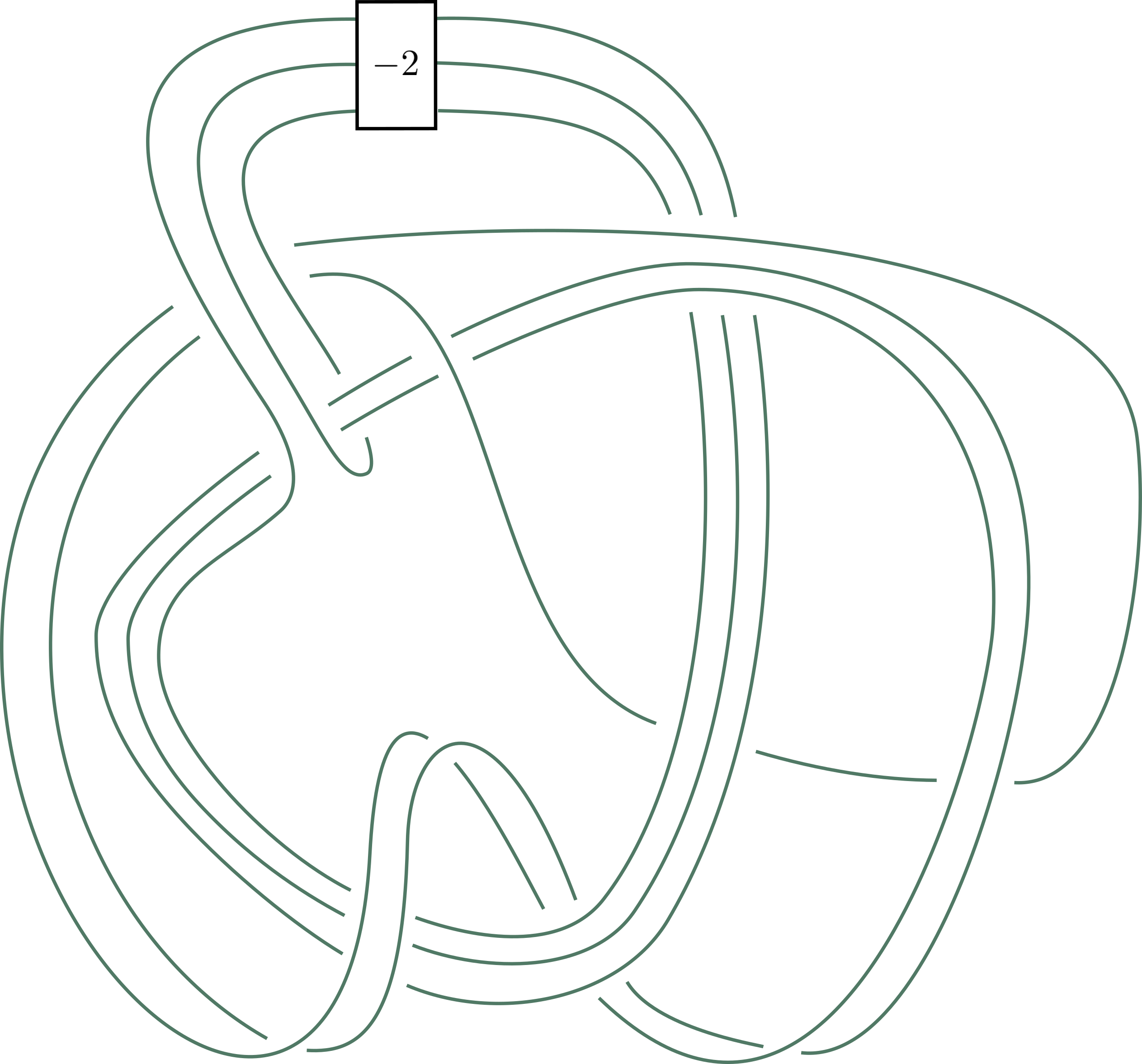}
\caption{The knot $K'$ shares a trace with the Conway knot}
\label{fig:J'}
\end{figure}

Thus it suffices to prove that $K'$ is not slice. The advantage of this perspective is that we do not have any reason to expect $K'$ is a mutant of a slice knot, so we hope that not all slice obstructions for $K'$ will vanish. This hope is complicated by the fact that $K'$ has Alexander polynomial 1, thus all its abelian and metabelian sliceness invariants vanish. In general it is difficult to distinguish knots with diffeomorphic traces in concordance (see \cite{Kir97}, problem 1.21), however, recent work of the author \cite{Pic18} and Miller and the author \cite{MP18} demonstrates that Rasmussen's s-invariant and the Heegaard Floer correction terms of the double branched covers can be used to distinguish such knots. Using the s-invariant, we show 
\begin{thm}\label{Thm:J'ns}
$K'$ is not slice.
\end{thm}

The s-invariant shares many formal properties with Ozsv{\'a}th-Szab{\'o}'s $\tau$ invariant. It is perhaps worth remarking then that for any $K'$ with $X(C)\cong X(K')$ one can show $\tau(K')=0$. 

In Section \ref{Sec:cons} we construct $K'$ and prove Proposition \ref{Prop:J'}. In Section \ref{Sec:obst} we compute $s(K')$ and prove Theorem \ref{Thm:J'ns}. We will assume familiarity with handle calculus, for details see \cite{GS99}.

\section{Acknowledgments}\label{Sec:ack}
\noindent The author was reminded of this problem during Shelly Harvey's talk at Bob Gompf's birthday conference. The author is in frequent conversation with Allison N. Miller, and those insightful conversations inform this work. The author is deeply indebted to her advisor John Luecke, whose constant encouragement, context, and insights are indispensable to her work. 

\section{Constructing K' which shares a trace with the Conway knot}\label{Sec:cons}
\noindent We begin by recalling the dualizable patterns construction, as presented in \cite{Pic18}. Let $L$ be a three component link with (blue, green, and red) components $B,G,$ and $R$ such that the following hold: the sublink $B\cup R$ is isotopic in $S^3$ to the link $B\cup \mu_B$ where $\mu_B$ denotes a meridian of $B$, the sublink $G\cup R$ is isotopic to the link $G\cup \mu_G$, and $lk(B,G)=0$. From $L$ we can define an associated four manifold $X$ by thinking of $R$ as a 1-handle, in dotted circle notation, and $B$ and $G$ as attaching spheres of 0-framed 2-handles. In a moment we will also define a pair of knots $K$ and $K'$ associated to $L$.
\begin{thm}\cite{Pic18} \label{Thm:diffeotraces}
$X\cong X(K)\cong X(K')$. 
\end{thm}

\begin{proof}
Isotope $L$ to a diagram in which $R$ has no self crossings (hence such that $R$ bounds a disk $D_R$ in the diagram) and in which $B\cap D_R$ is a single point. Slide $G$ over $B$ as needed to remove the intersections of $G$ with $D_R$. After the slides we can cancel the two handle with attaching circle $B$ with the one handle and we are left with a handle description for a 0-framed knot trace; this knot is $K'$. 

To construct $K$ and see $X\cong X(K)$, perform the above again with the roles of $B$ and $G$ reversed. 
\end{proof}

\begin{prop}\label{Prop:unknotting}
For any unknotting number 1 knot $K$, there exists a link $L$ as above and 4-manifold $X$ associated to $L$ as above so that $X\cong X(K)$. 
\end{prop}
\begin{figure}
\includegraphics[width=15cm]{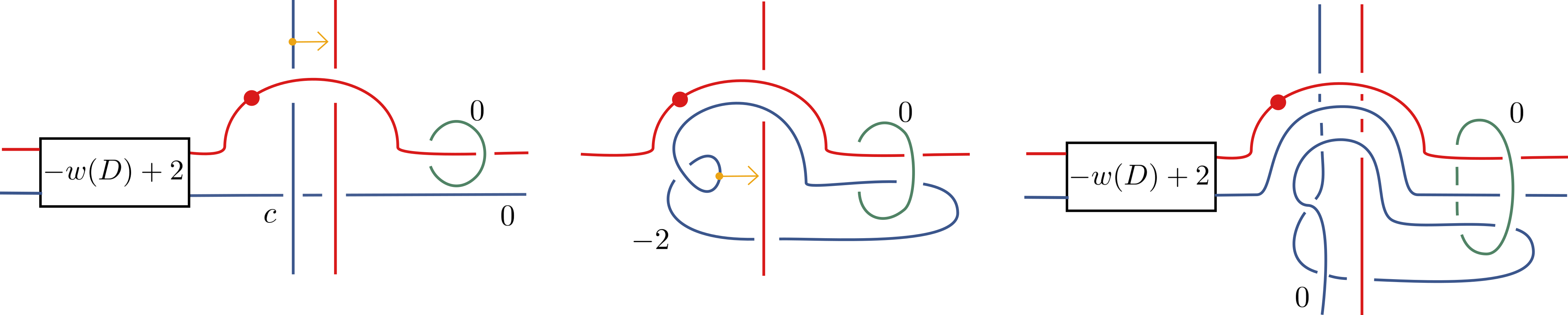}
\caption{Constructing a link $L$ associated to an unknotting number 1 knot $K$. Here $D$ denotes a (blue) diagram of $K$ with a positive unknotting crossing $c$, and $w(D)$ denotes the writhe of $D$.}
\label{fig:nearc1}
\end{figure}
\begin{proof}
Choose a (blue) diagram $D$ of $K$ with an unknotting crossing $c$. We will prove the claim for $c$ positive, the proof for $c$ negative is similar. Define knots $R$ and $G$ in $S^3\smallsetminus\nu(K)$ as in the left frame of Figure \ref{fig:nearc1}, where $R$ is a blackboard parallel of $D$ outside of the diagram. Define $X$ to be the four manifold obtained by thinking of $R$ as a one handle in dotted circle notation, and attaching 0-framed 2-handles along $K$ and $G$. Since $G$ and $R$ are a canceling 1-2 pair, we see that $X\cong X(K)$. It remains to construct a link $L$ presenting $X$, where $L$ satisfies the construction preceding Theorem \ref{Thm:diffeotraces}.

To this end, slide $K$ over $R$ as indicated in Figure \ref{fig:nearc1} to get a handle description for $X$ as in the center frame. Observe that the blue attaching sphere is isotopic to a meridian of $R$. As such, performing the indicated slide to get the right frame will yield a link $L$ with 0-framed blue attaching sphere $B$ which can be isotoped so that $B\cup R$ is isotopic to $B\cup \mu_B$, and one observes that $lk(B,G)=0$. 
\end{proof}

Thus for any unknotting number one knot $K$, one can produce a link $L$ as in Theorem \ref{Thm:diffeotraces}, and use $L$ to produce a knot $K'$ with $X(K')\cong X(K)$. We remark that the unknotting number one knot $K$ is in fact isotopic to the knot $K$ produced from $L$ as in the proof of Theorem \ref{Thm:diffeotraces}, though we wont rely on that here. 

\begin{proof}[Proof of Proposition \ref{Prop:J'}]
We now produce such a $K'$ for the Conway knot. We proceed as in the proofs of Proposition \ref{Prop:unknotting} and Theorem \ref{Thm:diffeotraces}; in order to produce a diagram of $K'$ with small crossing number we will perform additional isotopies throughout. See Figure \ref{Fig:whoa}.
\begin{figure}[h]
  \begin{subfigure}[b]{0.3\textwidth}
    \includegraphics[width=\textwidth]{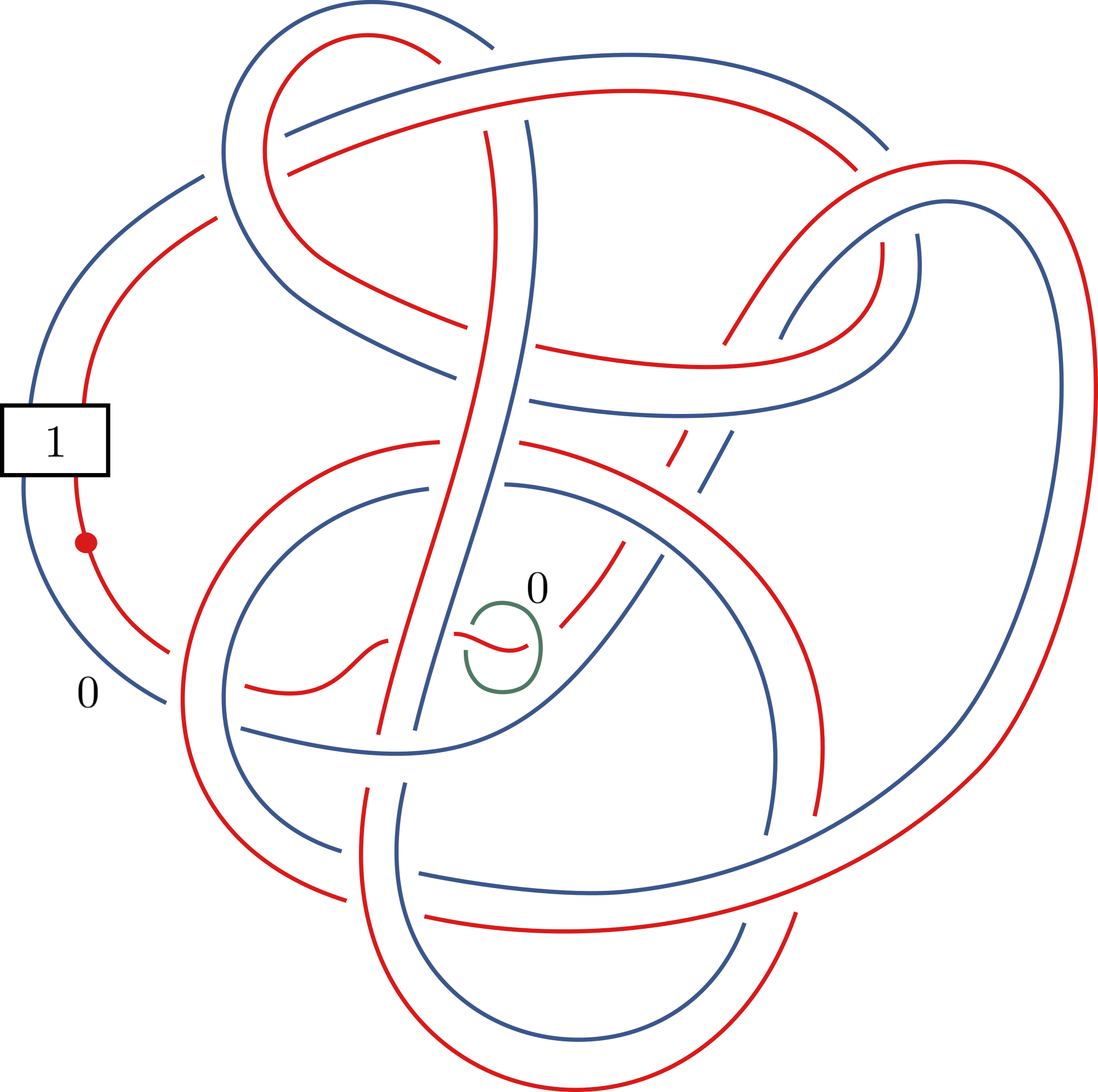}
    \caption{}
    \label{fig:1}
  \end{subfigure}
         \begin{subfigure}[b]{0.25\textwidth}
    \includegraphics[width=\textwidth]{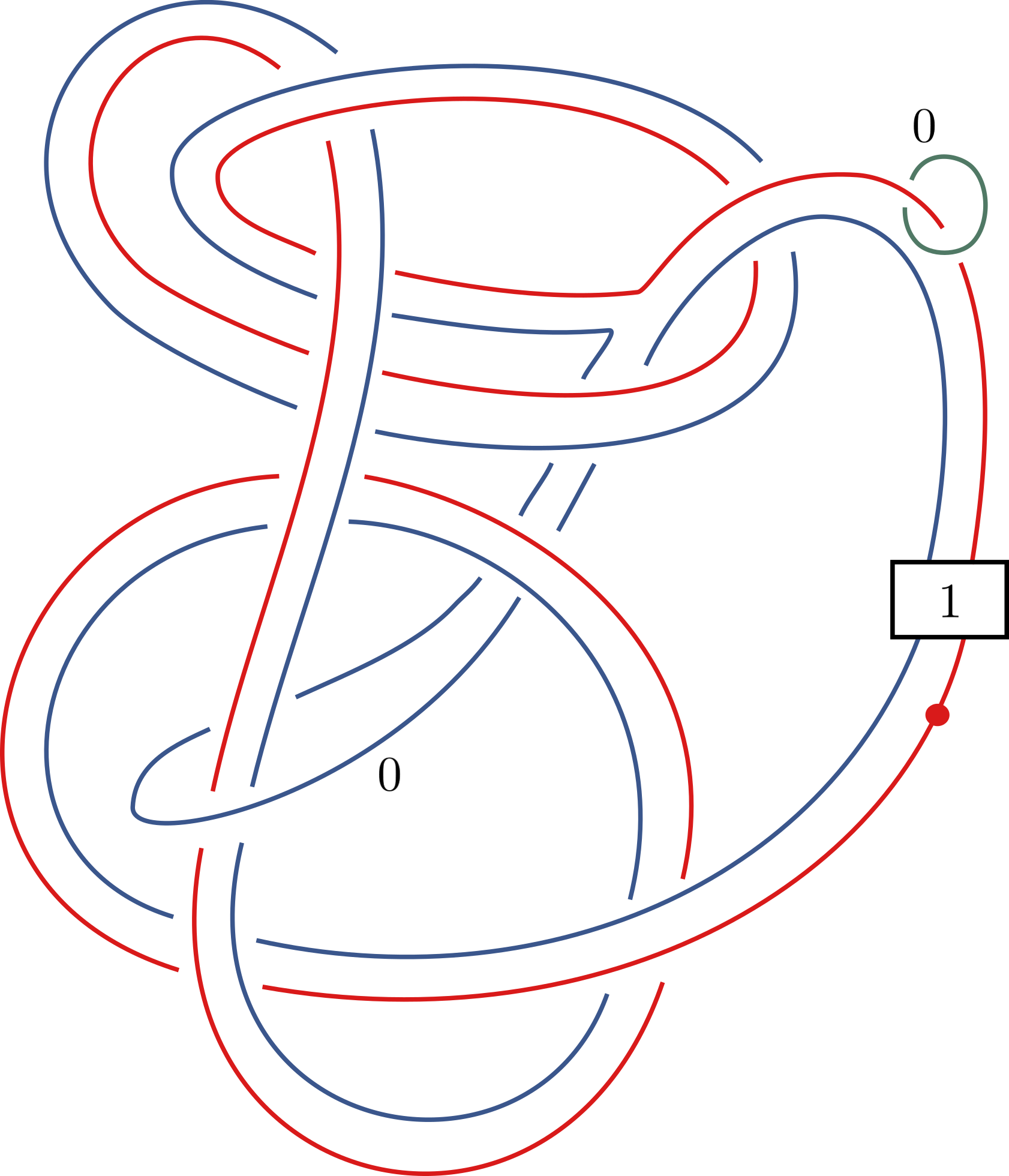}
    \caption{}
    \label{fig:2}
  \end{subfigure}
    \begin{subfigure}[b]{0.25\textwidth}
    \includegraphics[width=\textwidth]{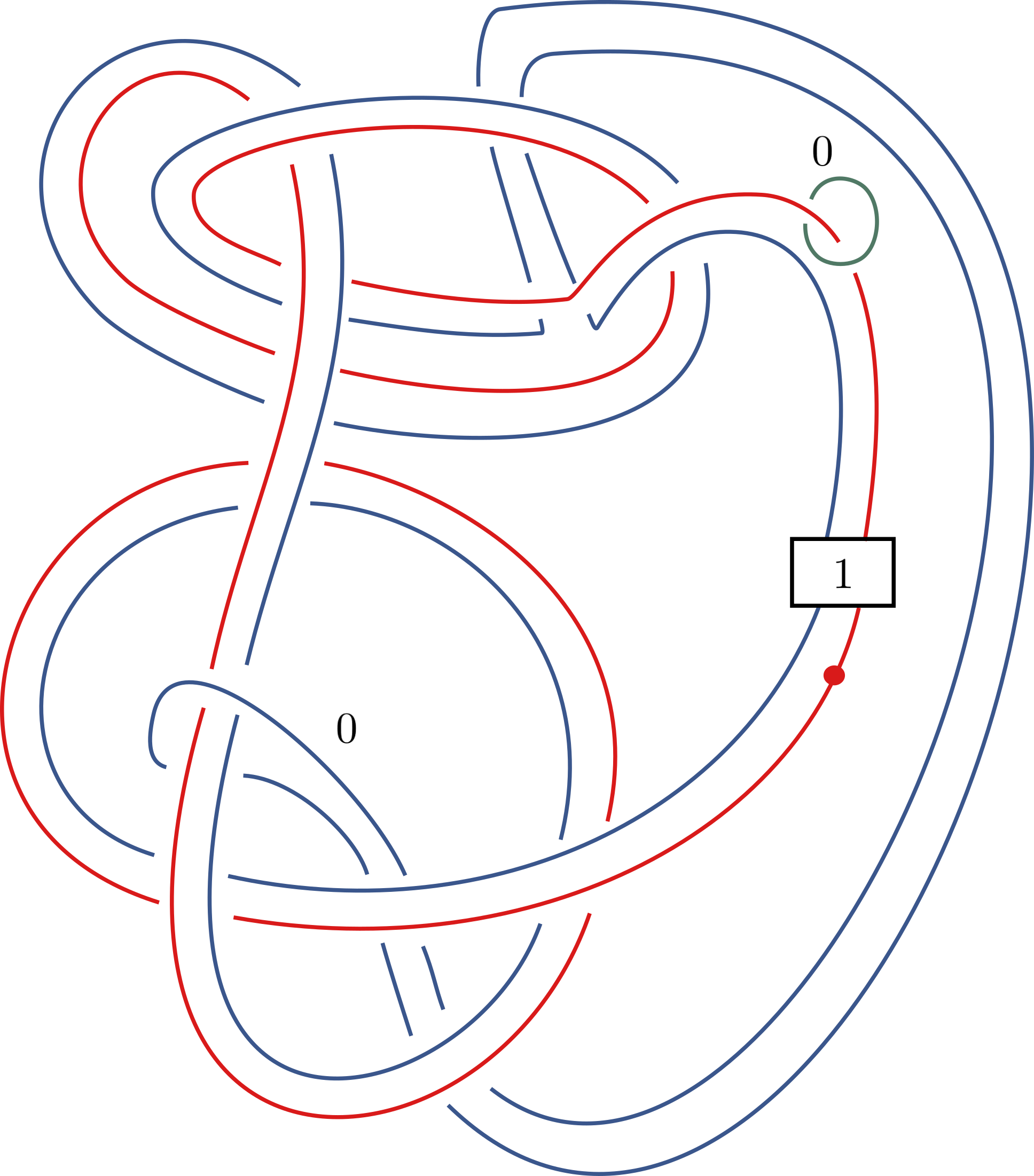}
    \caption{}
    \label{fig:3}
  \end{subfigure}
      \begin{subfigure}[b]{0.25\textwidth}
    \includegraphics[width=\textwidth]{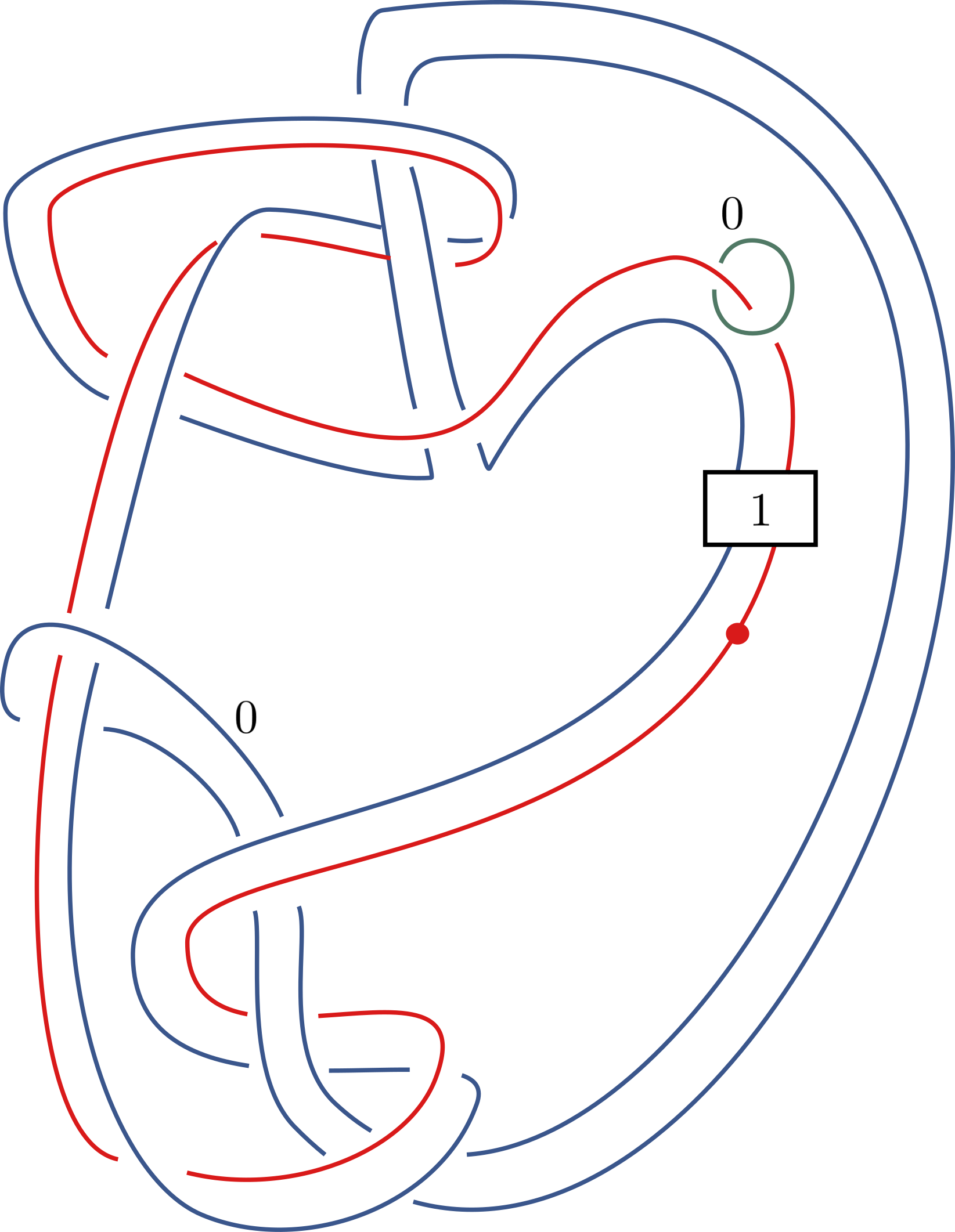}
    \caption{}
    \label{fig:4}
  \end{subfigure}
      \begin{subfigure}[b]{0.25\textwidth}
    \includegraphics[width=\textwidth]{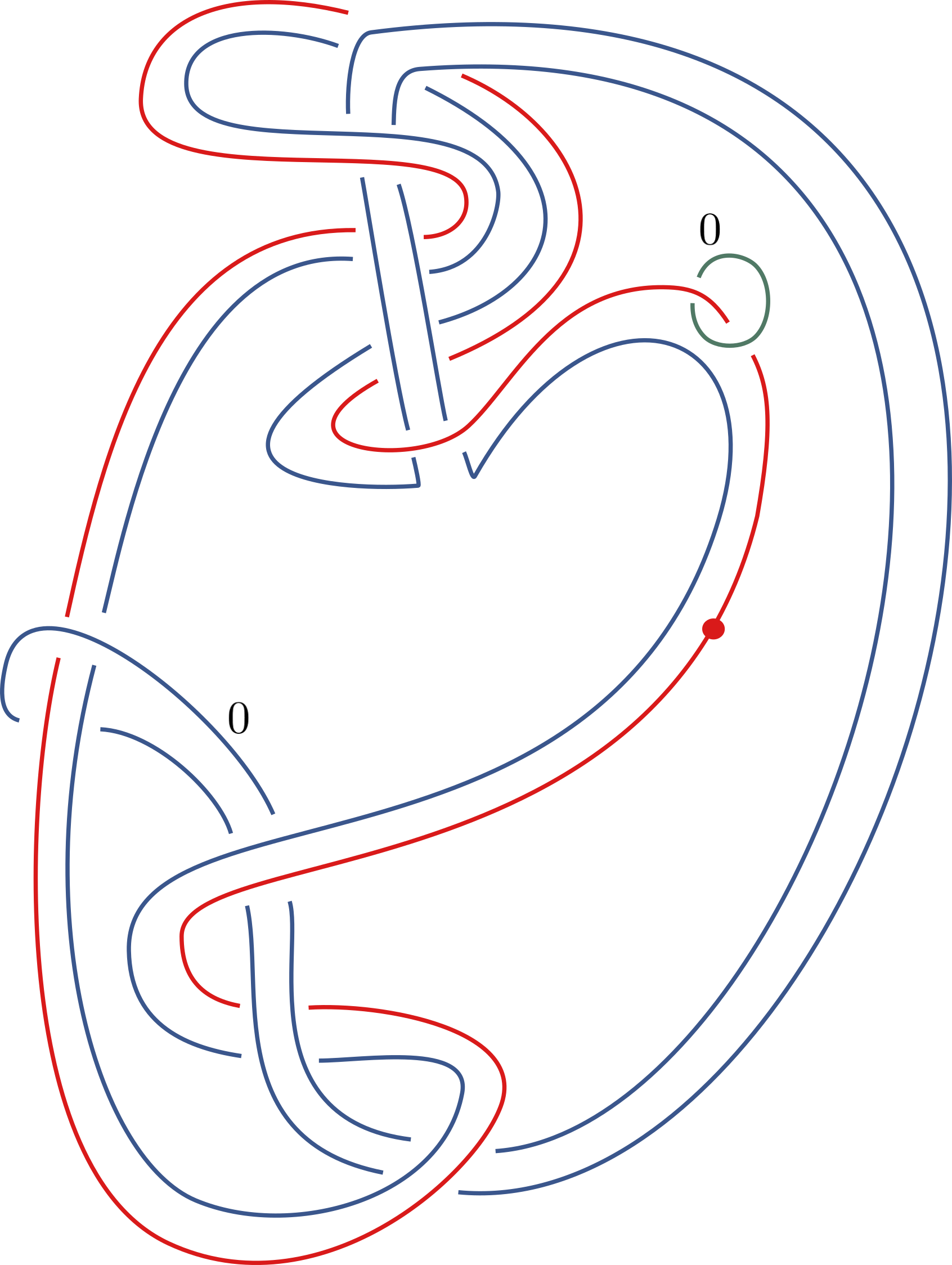}
    \caption{}
    \label{fig:5}
  \end{subfigure}  
        \begin{subfigure}[b]{0.3\textwidth}
    \includegraphics[width=\textwidth]{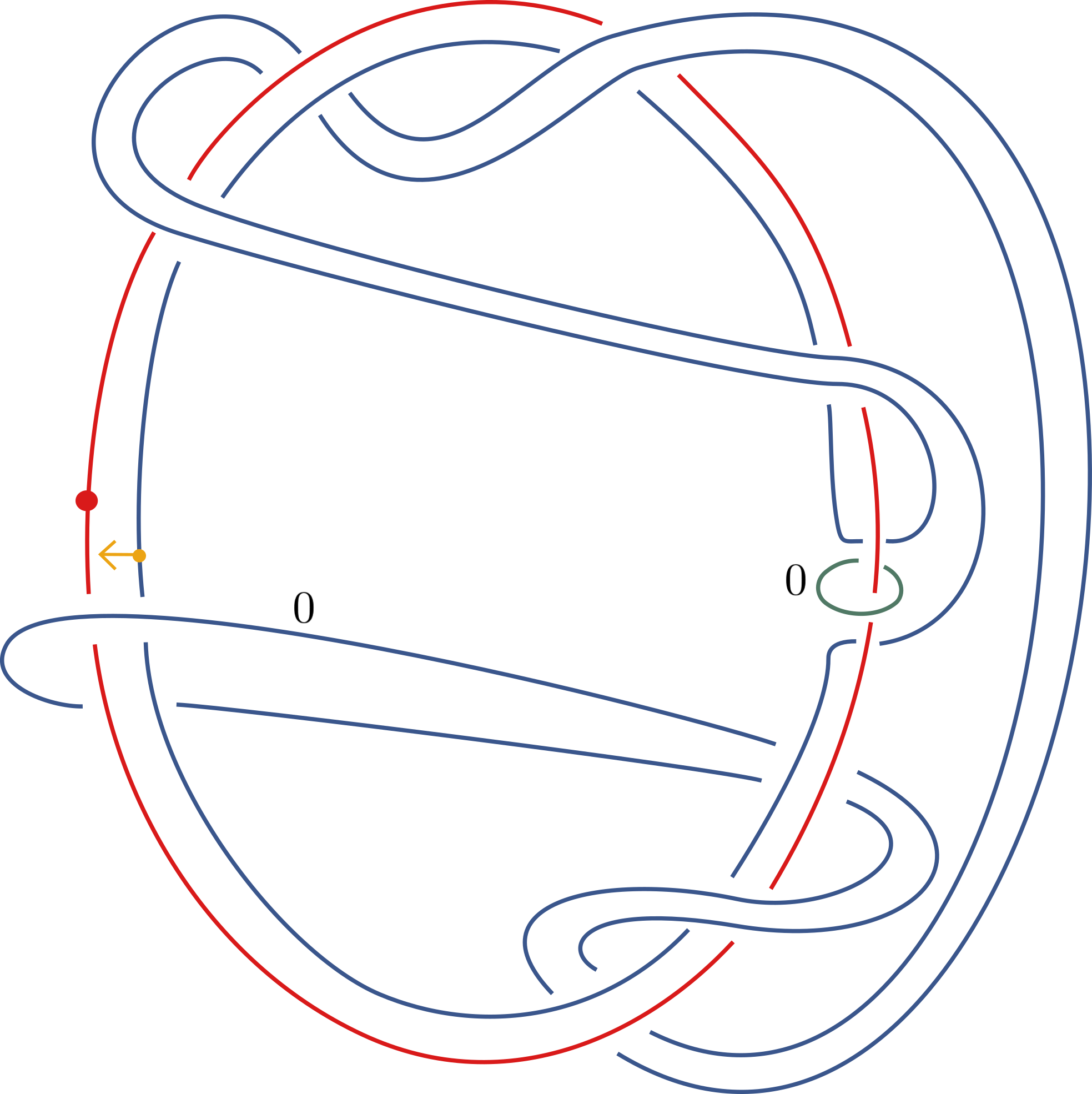}
    \caption{}
    \label{fig:6}
  \end{subfigure}
          \begin{subfigure}[b]{0.3\textwidth}
    \includegraphics[width=\textwidth]{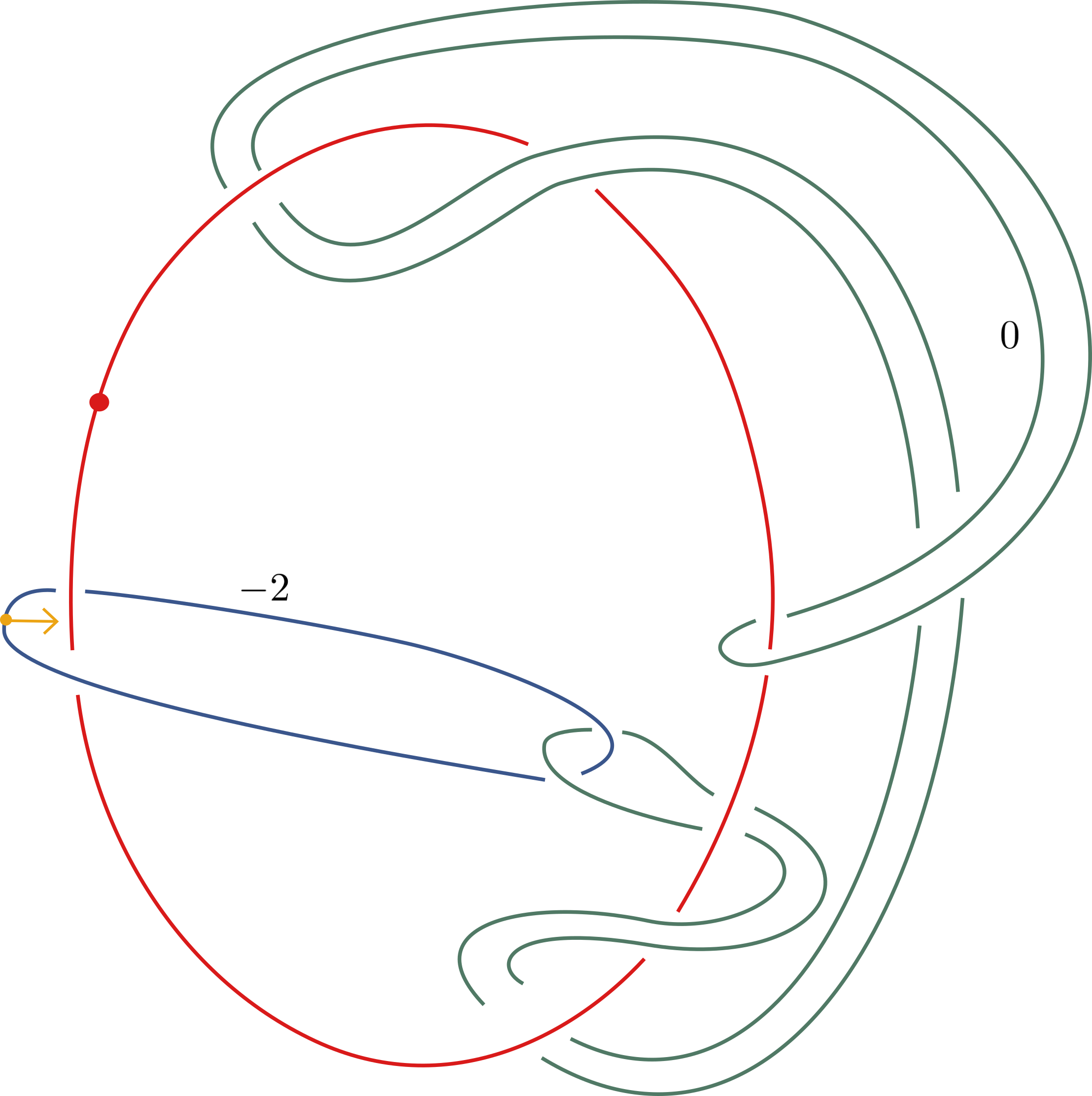}
    \caption{}
    \label{fig:8}
  \end{subfigure}
          \begin{subfigure}[b]{0.3\textwidth}
    \includegraphics[width=\textwidth]{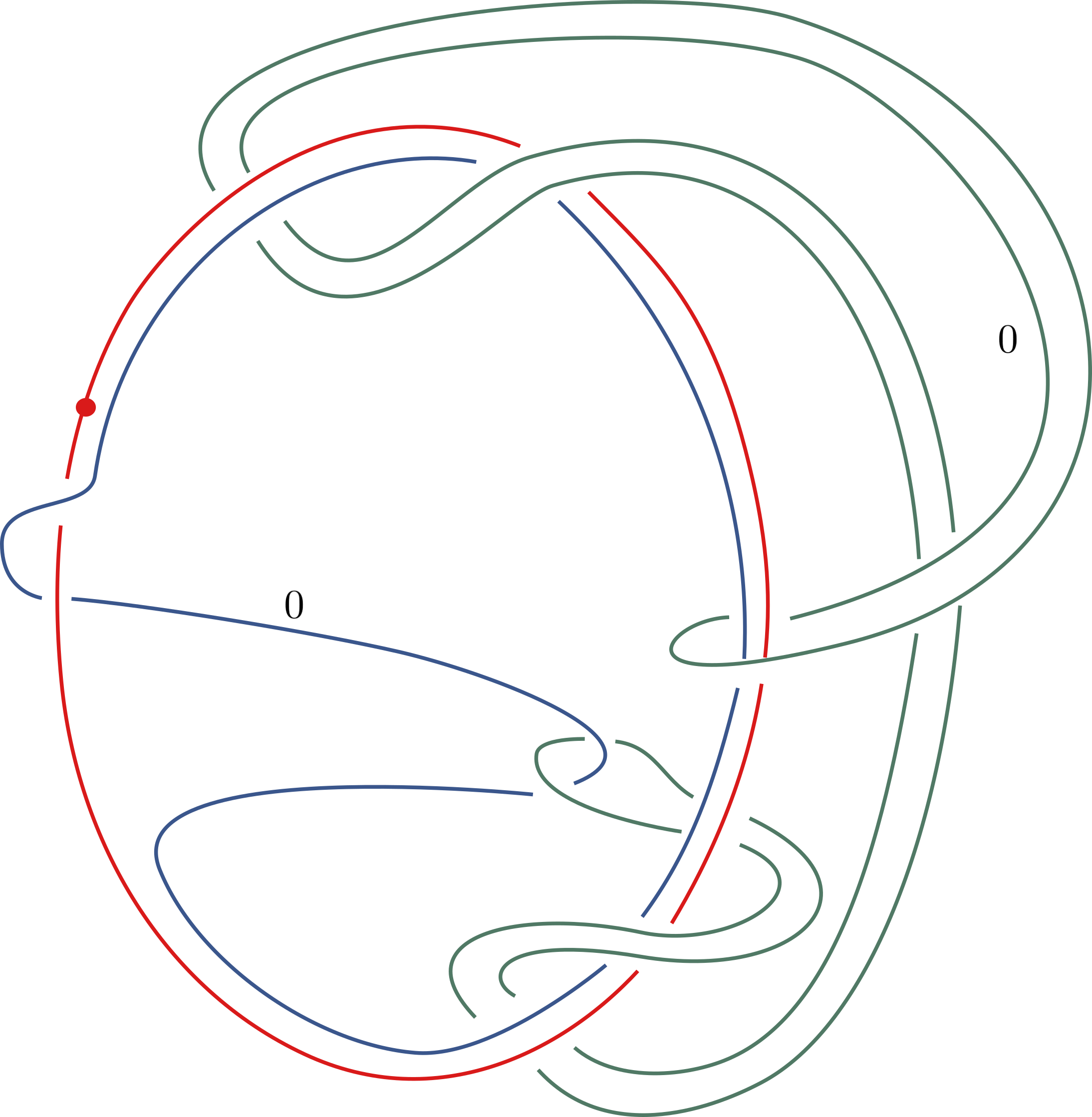}
    \caption{}
    \label{fig:85}
  \end{subfigure}
   \begin{subfigure}[b]{0.3\textwidth}
    \includegraphics[width=\textwidth]{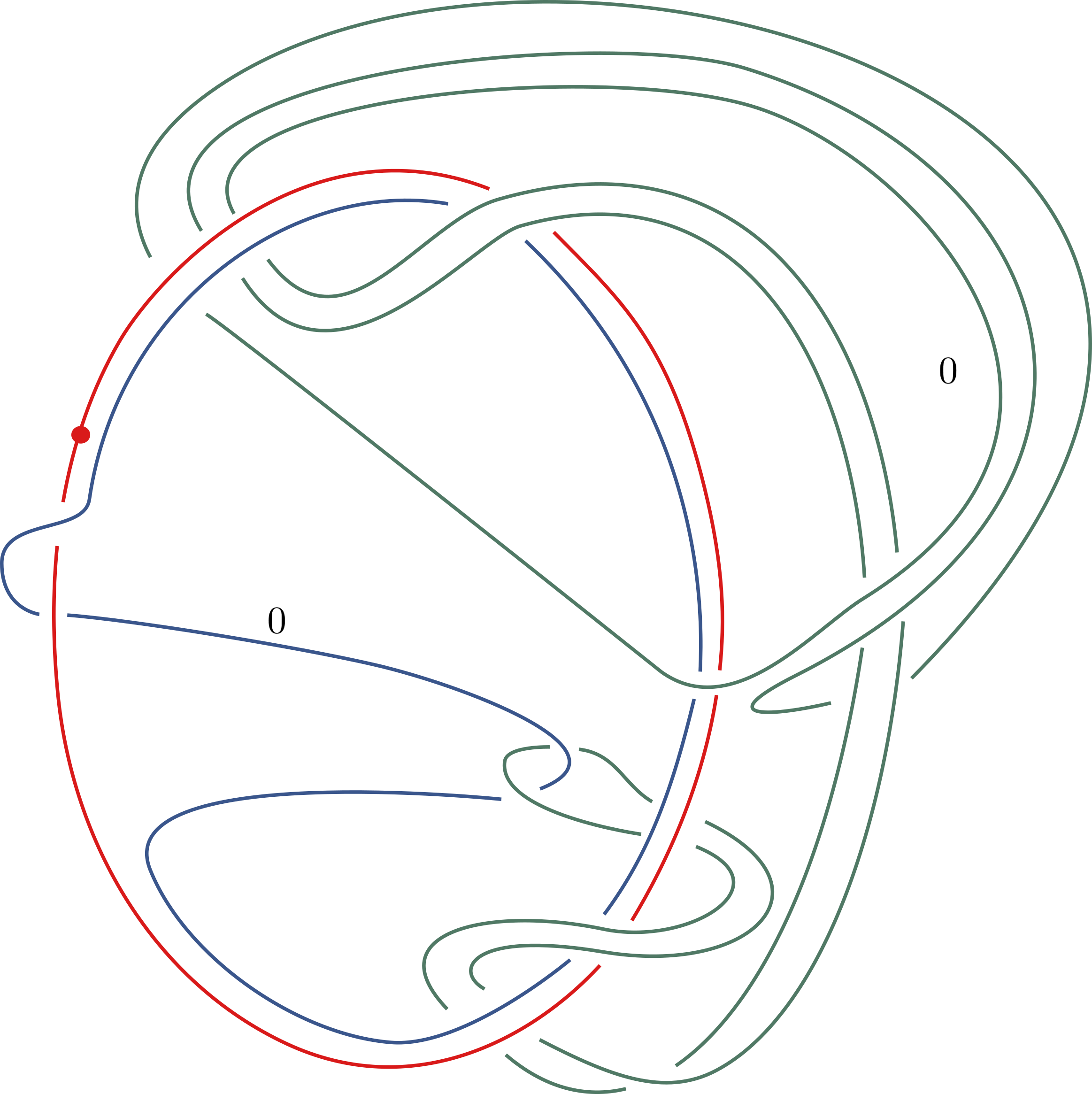}
    \caption{}
    \label{fig:9}
  \end{subfigure}
  \caption{Handle calculus exhibiting a diffeomorphism from $X(C)$ to $X(K')$ where $K'$ is the knot defined in Figure \ref{fig:J'}. Handle slides are denoted with arrows, the transition from $(L)$ to $(M)$ includes canceling a 1-2 pair, and all other changes are isotopies.\vspace{18pt}}\label{Fig:whoa}
  \end{figure}

  \begin{figure}\ContinuedFloat
         
\begin{subfigure}[b]{0.4\textwidth}
    \includegraphics[width=\textwidth]{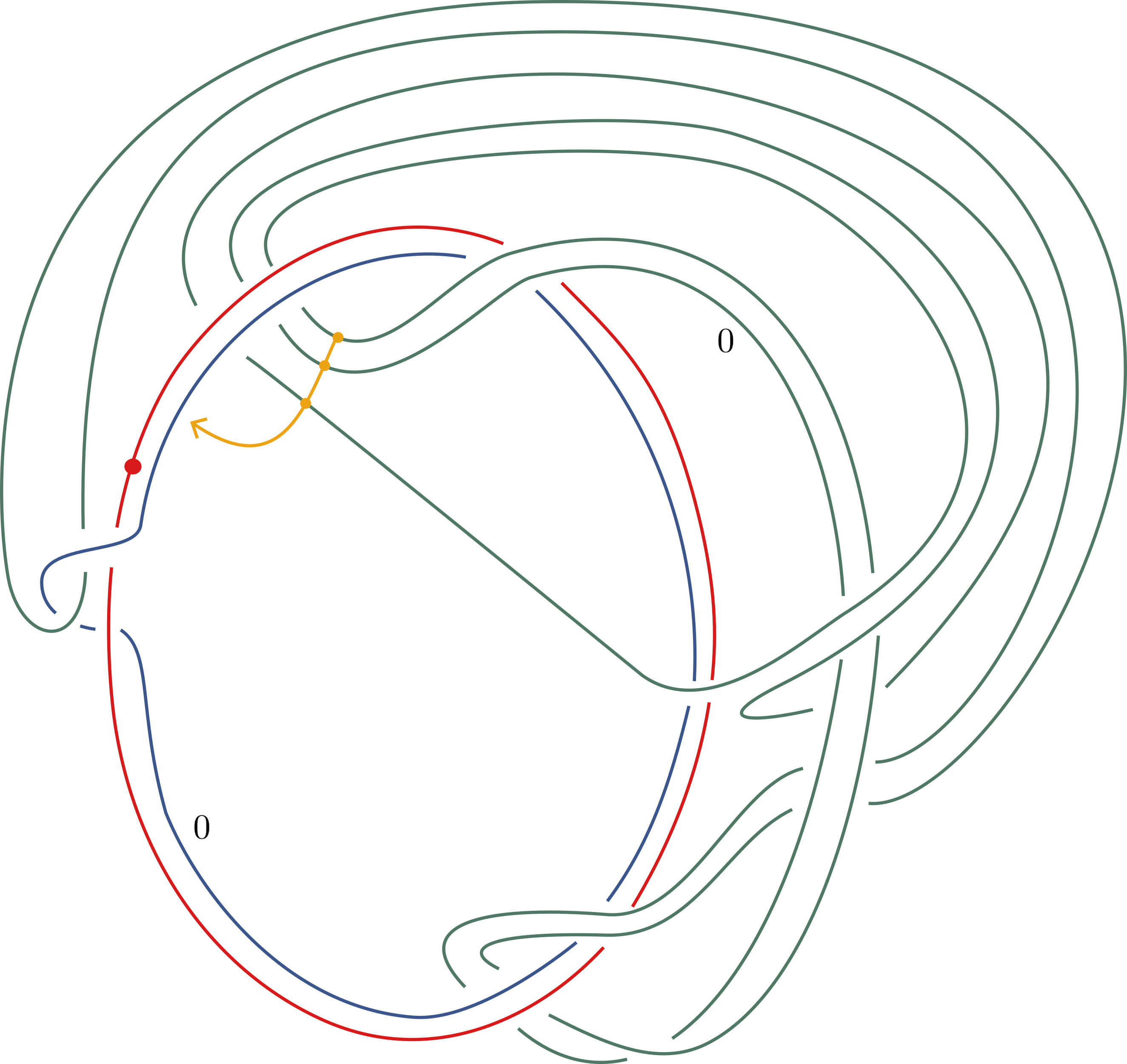}
    \caption{}
    \label{fig:10}
  \end{subfigure}
            \begin{subfigure}[b]{0.4\textwidth}
    \includegraphics[width=\textwidth]{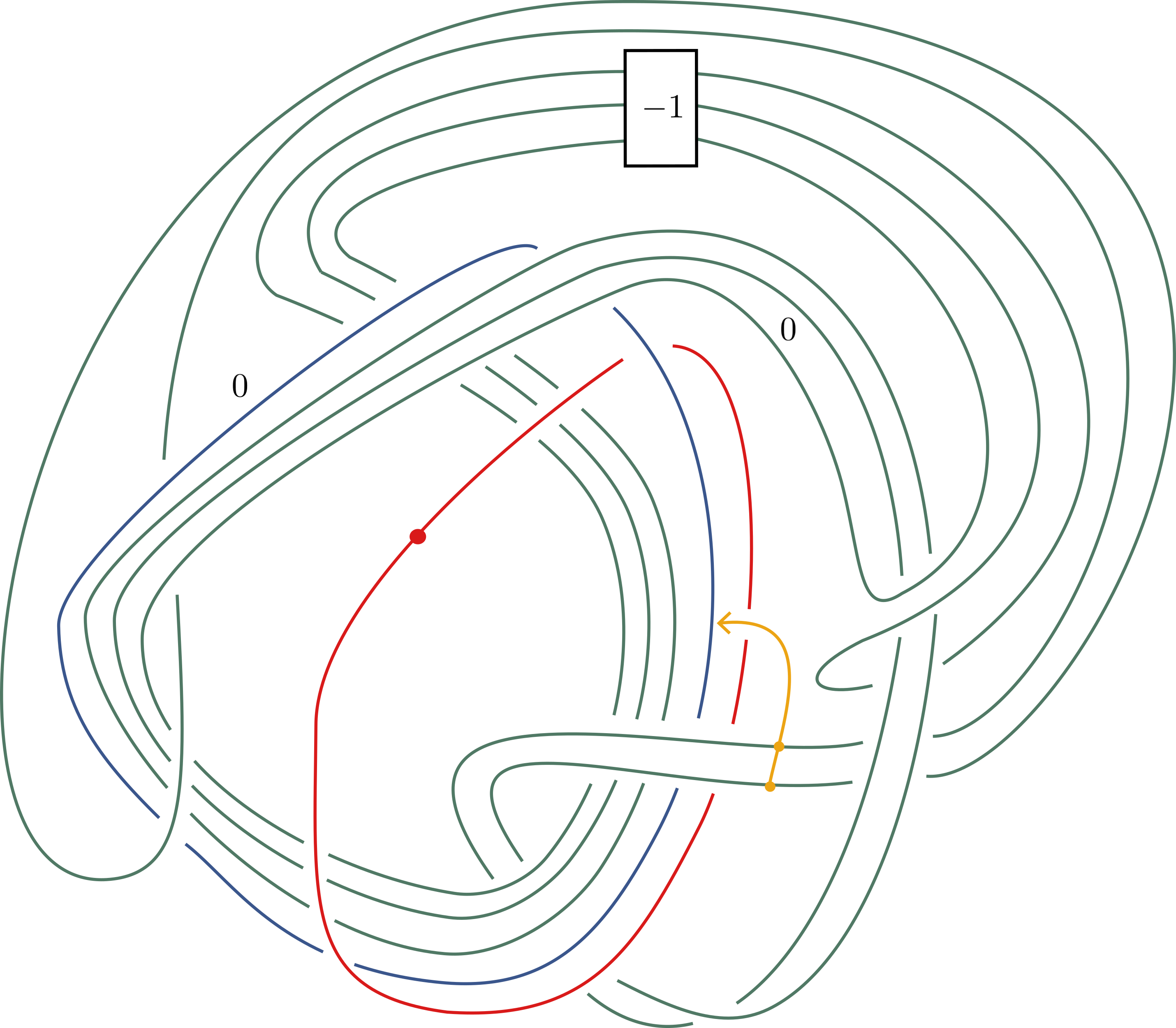}
    \caption{}
    \label{fig:11}
  \end{subfigure}
            \begin{subfigure}[b]{0.4\textwidth}
    \includegraphics[width=\textwidth]{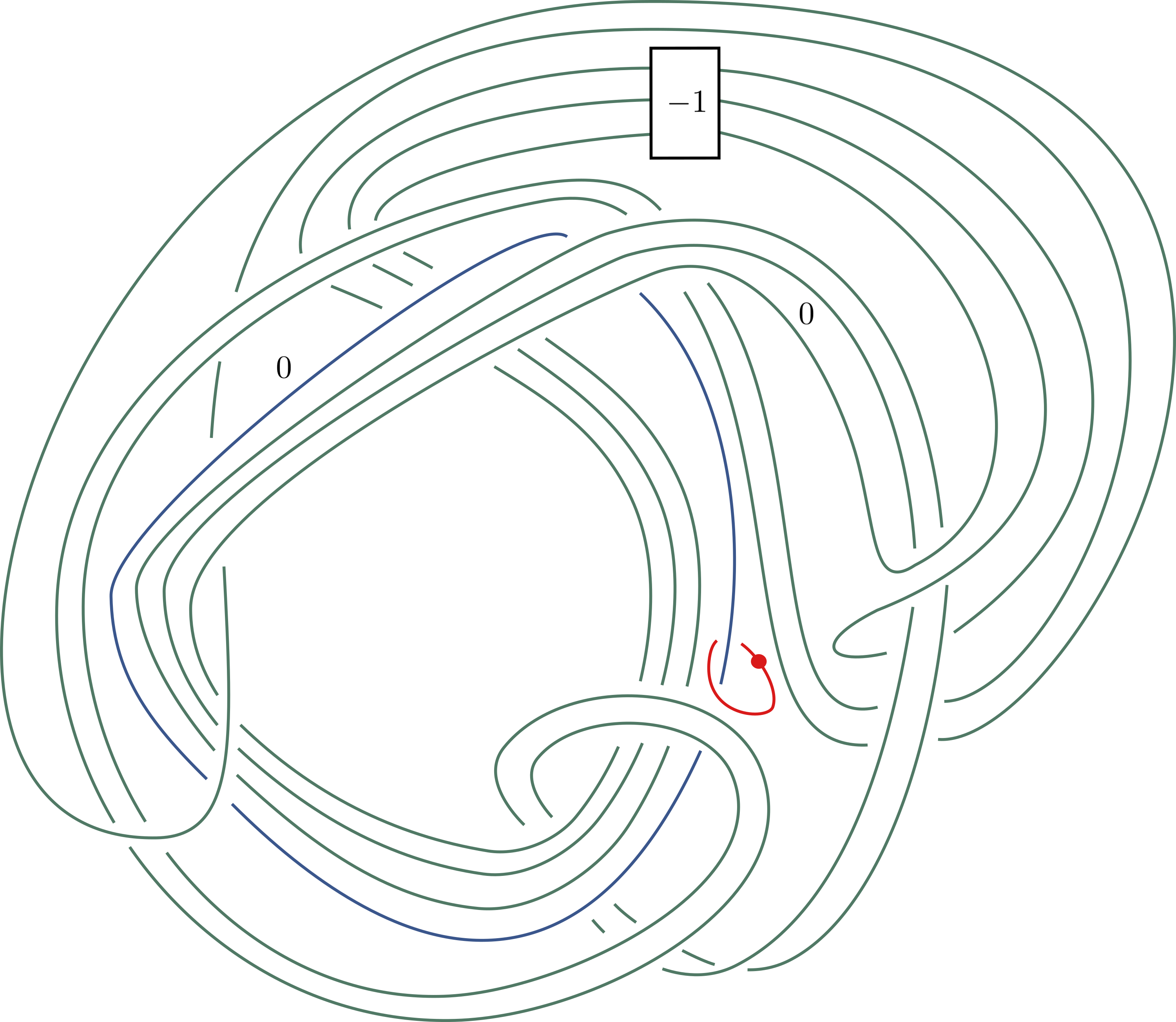}
    \caption{}
    \label{fig:12}
  \end{subfigure}
            \begin{subfigure}[b]{0.4\textwidth}
    \includegraphics[width=\textwidth]{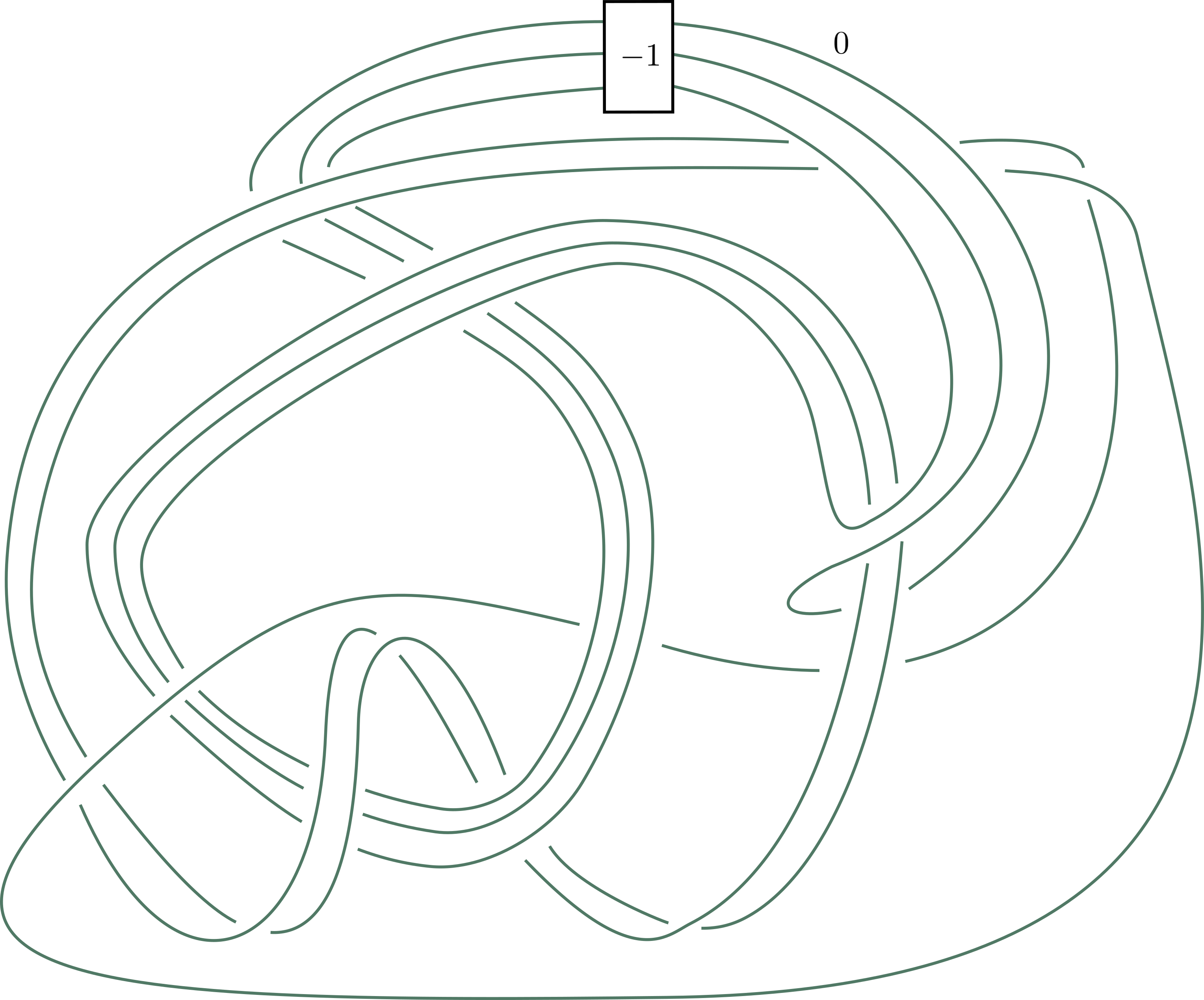}
    \caption{}
    \label{fig:13}
  \end{subfigure}
            \begin{subfigure}[b]{0.4\textwidth}
    \includegraphics[width=\textwidth]{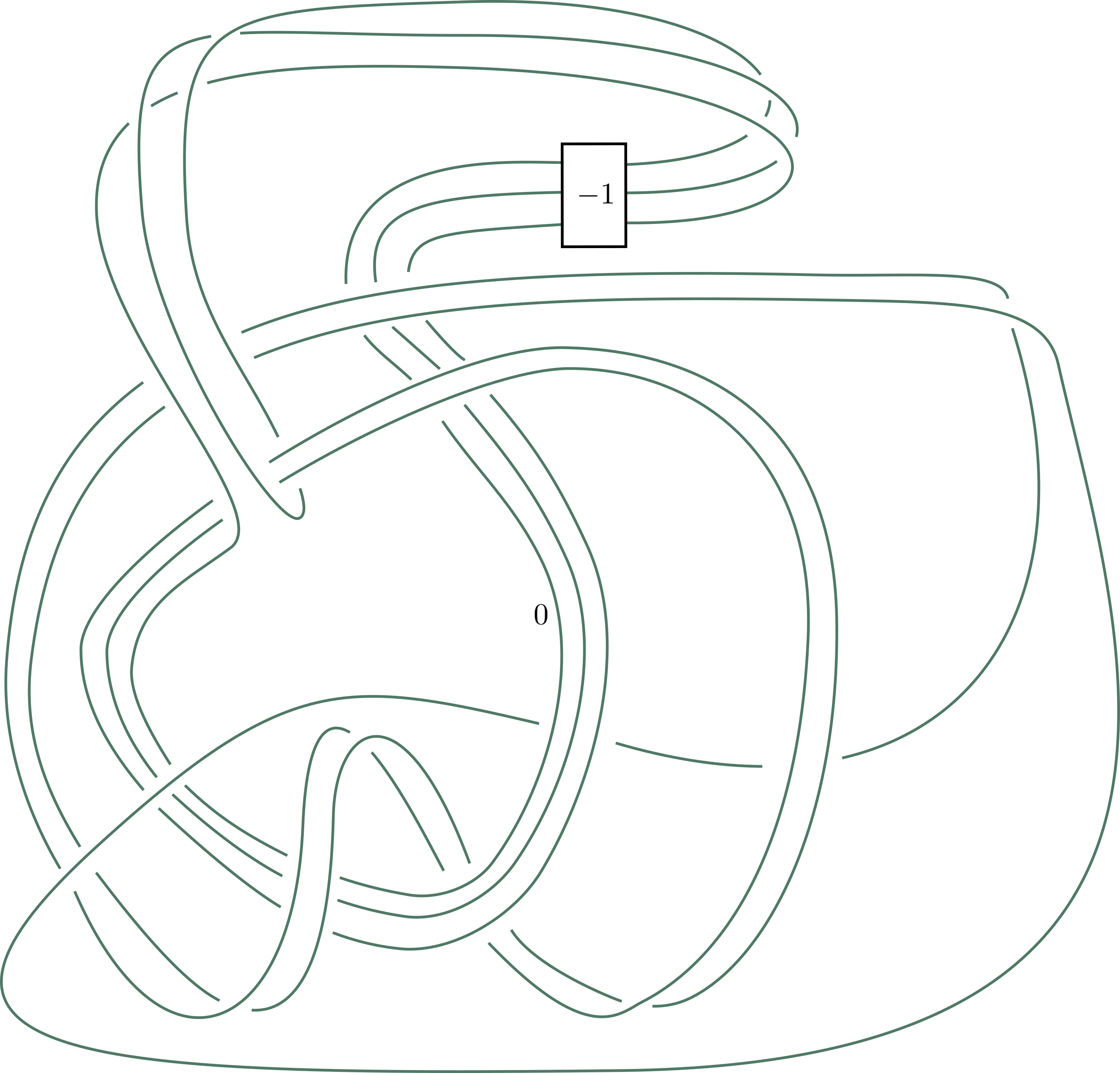}
    \caption{}
    \label{fig:14}
  \end{subfigure}
  \begin{subfigure}[b]{0.4\textwidth}
    \includegraphics[width=\textwidth]{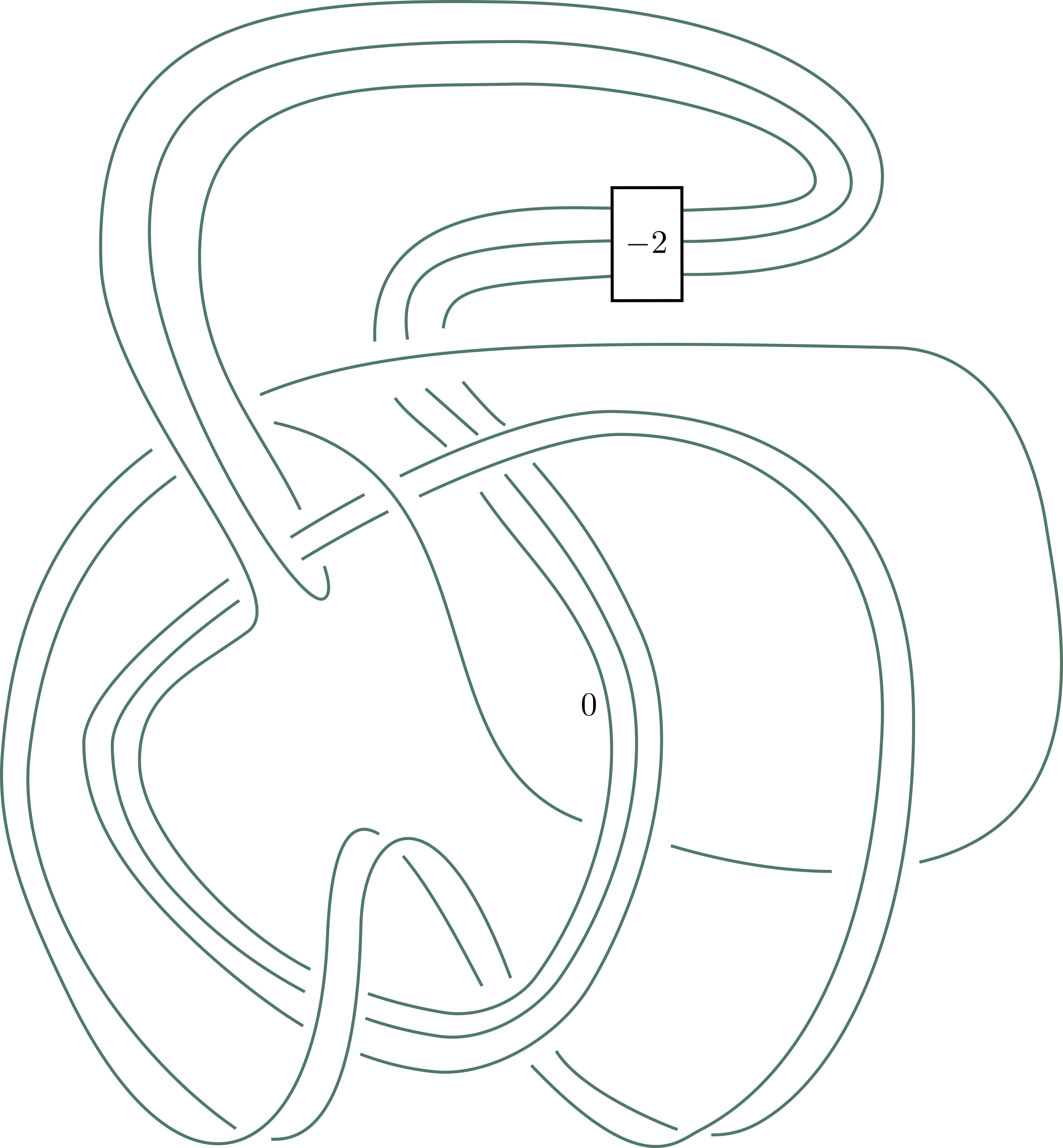}
    \caption{}
    \label{fig:15}
  \end{subfigure}
  \caption{continued}
\end{figure}
\end{proof}
\section{Showing K' is not slice}\label{Sec:obst}
In \cite{Kho00} Khovanov introduced a link invariant $Kh^{i,j}(L)$ which is the (co)homology of a finitely generated bigraded chain complex $(C^{i,j}(D_L),d)$. In our notation, $D_L$ denotes a diagram of $L$ and $i$ is refered to as the homological grading and $j$ the quantum grading. Later Lee \cite{Lee02} introduced a modification of the Khovanov differential: she considered instead a graded filtered complex $(C^{i,j}(D_L),d')$, such that $d'$ raises homological grading by 1 and for any homogenous $v\in C^{i,j}(D_L)$ the quantum grading of every monomial in in $d'(v)$ is greater than or equal to the quantum grading of $v$. As a consequence of her construction, there exists a spectral sequence with $(E^{i,j}_1(D_L), d_1)=(C^{i,j}(D_L),d)$ and $E_2^{i,j}=Kh^{i,j}(L)$ which converges to the homology of the Lee complex for $L$. We will denote this homlogy group $KhL^{i,j}(L)$. It will be relevant for us that the differentials $d_n$ of th spectral sequence have bidegree $(1,4(n-1))$ (see \cite{Ras10} or \cite{Shu07}). Lee proves that for any knot $K$ $KhL(K)=\mathbb{Q}\oplus\mathbb{Q}$ where both generators are located in grading $i=0$. Rasmussen used this to define an integer valued knot invariant $s(K)$ as follows.

\begin{thm}[\cite{Ras10}]
For any knot $K$ the generators of Lee homology are located in gradings $(i,j)=(0,s(K)\pm 1)$. If $K$ is slice then $s(K)=0$. 
\end{thm}

\begin{proof}[Proof of Theorem \ref{Thm:J'ns}]
Let $K'$ be the knot from Proposition \ref{Prop:J'}; to show $K'$ is not slice we will calculate $s(K')$. To begin, we compute the Khovanov homology of $K'$, using Bar-Natan's Fast-Kh routines available at \cite{KAT}. These routines produce the polynomial $Kh(K)(t,q):=\Sigma_{i,j}t^iq^j\text{rank}(Kh^{i,j}(K)\otimes\mathbb{Q})$. We plot the values $\text{rank}(Kh^{i,j}(K')\otimes\mathbb{Q})$ in Table \ref{Tab:Kh}. 

\begin{table}
\resizebox{\columnwidth}{!}{%
\begin{tabular}{c||c|c|c|c|c|c|c|c|c|c|c|c|c|c|c|c|c|c|c|c|c|c|c|c|c|c|c|c|c|c|c|c|c|c|c|c}
& -3 & -2 & -1 & 0 & 1 & 2 & 3 & 4 & 5 & 6 & 7 & 8 & 9 & 10 & 11 & 12 & 13 & 14 & 15 & 16 & 17 & 18 & 19 & 20 & 21 & 22 & 23 & 24 & 25 & 26 & 27 & 28 & 29 & 30 & 31 & 32 \\
 \hline\hline
49&  &  &  &  &  &  &  &  &  &  &  &  &  &  &  &  &  &  &  &  &  &  &  &  &  &  &  &  &  &  &  &  &  &  &  & 1 \\\hline
47&  &  &  &  &  &  &  &  &  &  &  &  &  &  &  &  &  &  &  &  &  &  &  &  &  &  &  &  &  &  &  &  &  &  &  &  \\\hline
45&  &  &  &  &  &  &  &  &  &  &  &  &  &  &  &  &  &  &  &  &  &  &  &  &  &  &  &  &  &  &  &  & 1 & 1 & 1 &  \\\hline
43&  &  &  &  &  &  &  &  &  &  &  &  &  &  &  &  &  &  &  &  &  &  &  &  &  &  &  &  &  &  &  & 1 & 1 &  &  &  \\\hline
41&  &  &  &  &  &  &  &  &  &  &  &  &  &  &  &  &  &  &  &  &  &  &  &  &  &  &  &  &  & 1 & 1 & 1 & 1 &  &  &  \\\hline
39&  &  &  &  &  &  &  &  &  &  &  &  &  &  &  &  &  &  &  &  &  &  &  &  &  &  &  & 1 & 1 & 1 & 2 & 1 &  &  &  &  \\\hline
37&  &  &  &  &  &  &  &  &  &  &  &  &  &  &  &  &  &  &  &  &  &  &  &  &  &  & 1 & 1 & 2 & 1 &  &  &  &  &  &  \\\hline
35&  &  &  &  &  &  &  &  &  &  &  &  &  &  &  &  &  &  &  &  &  &  &  &  & 1 & 1 & 3 & 2 & 1 & 1 &  &  &  &  &  &  \\\hline
33&  &  &  &  &  &  &  &  &  &  &  &  &  &  &  &  &  &  &  &  &  &  &  & 2 & 3 & 2 & 1 & 1 &  &  &  &  &  &  &  &  \\\hline
31&  &  &  &  &  &  &  &  &  &  &  &  &  &  &  &  &  &  &  &  &  & 1 & 2 & 3 & 2 & 2 & 1 &  &  &  &  &  &  &  &  &  \\\hline
29&  &  &  &  &  &  &  &  &  &  &  &  &  &  &  &  &  &  &  &  & 2 & 2 & 4 & 4 & 1 &  &  &  &  &  &  &  &  &  &  &  \\\hline
27&  &  &  &  &  &  &  &  &  &  &  &  &  &  &  &  &  &  & 1 & 2 & 4 & 4 & 2 & 1 &  &  &  &  &  &  &  &  &  &  &  &  \\\hline
25&  &  &  &  &  &  &  &  &  &  &  &  &  &  &  &  &  & 2 & 3 & 5 & 3 & 2 & 1 &  &  &  &  &  &  &  &  &  &  &  &  &  \\\hline
23&  &  &  &  &  &  &  &  &  &  &  &  &  &  &  & 1 & 4 & 4 & 3 & 4 & 2 &  &  &  &  &  &  &  &  &  &  &  &  &  &  &  \\\hline
21&  &  &  &  &  &  &  &  &  &  &  &  &  &  & 2 & 3 & 4 & 4 & 3 & 1 &  &  &  &  &  &  &  &  &  &  &  &  &  &  &  &  \\\hline
19&  &  &  &  &  &  &  &  &  &  &  &  &  & 3 & 4 & 6 & 4 & 1 & 1 &  &  &  &  &  &  &  &  &  &  &  &  &  &  &  &  &  \\\hline
17&  &  &  &  &  &  &  &  &  &  &  & 1 & 4 & 6 & 4 & 2 & 1 &  &  &  &  &  &  &  &  &  &  &  &  &  &  &  &  &  &  &  \\\hline
15&  &  &  &  &  &  &  &  &  &  & 2 & 4 & 5 & 3 & 2 & 1 &  &  &  &  &  &  &  &  &  &  &  &  &  &  &  &  &  &  &  &  \\\hline
13&  &  &  &  &  &  &  &  &  & 3 & 4 & 5 & 4 & 1 &  &  &  &  &  &  &  &  &  &  &  &  &  &  &  &  &  &  &  &  &  &  \\\hline
11&  &  &  &  &  &  &  &  & 3 & 5 & 4 & 2 &  &  &  &  &  &  &  &  &  &  &  &  &  &  &  &  &  &  &  &  &  &  &  &  \\\hline
9&  &  &  &  &  &  & 1 & 3 & 5 & 3 & 1 &  &  &  &  &  &  &  &  &  &  &  &  &  &  &  &  &  &  &  &  &  &  &  &  &  \\\hline
7&  &  &  &  &  & 2 & 3 & 3 & 3 &  &  &  &  &  &  &  &  &  &  &  &  &  &  &  &  &  &  &  &  &  &  &  &  &  &  &  \\\hline
5&  &  &  &  & 1 & 3 & 3 & 2 &  &  &  &  &  &  &  &  &  &  &  &  &  &  &  &  &  &  &  &  &  &  &  &  &  &  &  &  \\\hline
3&  &  &  & 1 & 3 & 3 &  &  &  &  &  &  &  &  &  &  &  &  &  &  &  &  &  &  &  &  &  &  &  &  &  &  &  &  &  &  \\\hline
1&  &  & 2 & 2 & 2 &  &  &  &  &  &  &  &  &  &  &  &  &  &  &  &  &  &  &  &  &  &  &  &  &  &  &  &  &  &  &  \\\hline
-1&  & 1 &  & 1 &  &  &  &  &  &  &  &  &  &  &  &  &  &  &  &  &  &  &  &  &  &  &  &  &  &  &  &  &  &  &  &  \\\hline
-3&  & 2 &  &  &  &  &  &  &  &  &  &  &  &  &  &  &  &  &  &  &  &  &  &  &  &  &  &  &  &  &  &  &  &  &  &  \\\hline
-5& 1 &  &  &  &  &  &  &  &  &  &  &  &  &  &  &  &  &  &  &  &  &  &  &  &  &  &  &  &  &  &  &  &  &  &  &  \\
\end{tabular}%
}\caption{}\label{Tab:Kh}
\end{table}

Since the Lee homology is supported in grading $i=0$, we see that $s(K')\in\{0,2\}$. To demonstrate that in fact $s(K')=2$ we will use the fact that all higher differentials in the spectral sequence to the Lee homology have bidegree $(1,4(n-1))$. Consider a generator $x$ of $Kh^{0,3}(K')$. If $x$ were to die on the $n^{th}$ page of the spectral sequence ($n\ge 2$) we would have to have that either $d_n(x)\neq 0$ or there exists a $y$ with $d_n(y)=x$. Since $Kh^{i,j}(K')$ has no generators in gradings $\{1,4(n-1)+3\}$ or $\{-1,-4(n-1)-3\}$ for any $n\ge 2$, neither of these can happen. As such, $x$ survives to the $E^\infty$ page, and $s(K')=2$. 
\end{proof}

\bibliography{satellite}
\end{document}